\documentclass[12pt]{article}

\usepackage[T1]{fontenc} % Use 8-bit encoding that has 256 glyphs

\usepackage[utf8]{inputenc} % Required for including letters with accents

\usepackage{amsmath,amssymb,amsthm} % For including math equations, theorems, symbols, etc

\usepackage{varioref} % More descriptive referencing

\usepackage{datetime}

\usepackage{pifont}% for ding{112}

\newcommand{\li}{\ell_\infty}
\newcommand{\e}{e}
\renewcommand{\sectionmark}[1]{\markright{{\jobname{\textrm.tex}\ \ \ \ \ \ \ \ \ \today\ \ \ \ }}} % The header for all pages (oneside) or for even pages (twoside)
\usepackage{epsfig}
\usepackage{amssymb}
\usepackage{graphicx}

\usepackage[usenames,dvipsnames]{color}
\usepackage{setspace}%this stops lists having so much gap in them
\usepackage{enumitem}
\setlist{nolistsep}

\usepackage[mathscr]{eucal}%Euler script, use \mathscr

\usepackage{pifont}% for ding{112}

\newcommand{\Z}{\mathbb{Z}}
\newcommand{\F}{{\mbox F}}

\newcommand{\PG}{{\textup{PG}}}

\newcommand{\Fqq}{\mathbb{F}_{q^2}}

\newcommand{\Fq}{\mathbb{F}_{q}}

\newcommand\rd{{h}}

\newcommand\U{{\mathcal U}}
\newcommand\Uab{{\mathcal U}_{a,b}}
\newtheorem{theorem}{Theorem}

\newtheorem{lemma}[theorem]{Lemma}
\newtheorem{corollary}[theorem]{Corollary}
\newtheorem{conjecture}[theorem]{{{Conjecture}}}
\newtheorem{defn}[theorem]{Definition}
\theoremstyle{remark} % Define theorem styles here based on the remark style (used for remarks and notes)

\newcommand{\Label}{\label}
\newcommand{\Labele}{\label}

\usepackage{pdfpages}

\begin{document}

\title{Triple O'Nan Configurations in Buekenhout-Metz Unitals of Odd Order}

\author{Wen-Ai Jackson\footnote{School of Mathematical Sciences, University of Adelaide, Adelaide 5005, Australia, Email:\tt{wen.jackson@adelaide.edu.au}}\ \  and Peter Wild\footnote{Royal Holloway, University of London, Egham TW20 0EX, Email:\tt{peterrwild@gmail.com}}}

%\date{}
\maketitle

%{{\color{RubineRed}\jobname{\rm.tex}, \today\ at  \currenttime}}
\begin{abstract}An O'Nan configuration~\cite{ONan1972}  in a unital is a set of four lines forming a quadrilateral, where the six intersections of pairs of lines are points of the unital.  In 2019 Feng and Li~\cite{FengLi2019} elegantly construct O'Nan configurations in Buekenhout-Metz unitals, in particular, for odd order unitals. We extend their work by showing the existence of Triple O'Nan configurations (a configuration containing three distinct O'Nan configurations) in these odd order Buekenhout-Metz unitals.
 \end{abstract}

%{\color{Purple}\begin{spacing}{0} \tableofcontents \end{spacing}}

%===================================================================================

\section{Introduction}\Label{sec:intro}

A \emph{unital of order $n$} is a 2--$(n^3+1,n+1,1)$ design. Unitals of prime power order $q$ exist in the Desarguesian projective planes $\PG(2,q^2)$, as the absolute points and lines of a hermitian polarity. Such unitals are called \emph{classical} unitals. Buekenhout \cite{Buekenhout1976} constructed unitals in 2--dimensional translation planes. This was generalised by Metz \cite{Metz1979} to construct non-classical unitals in $\PG(2,q^2)$. These unitals are called Buekenhout-Metz (BM) unitals. They have the property that every line of $\PG(2,q^2)$ meets the unital in 1 or $q+1$ points. Further, the unital has exactly one point on the line at infinity, called the \emph{special} point. For a comprehensive treatise on unitals, see Barwick and Ebert \cite{thebook}.

An \emph{O'Nan configuration} in a unital is a set of four distinct lines forming a quadrilateral, whose six points of intersection lie in the unital.  In 1972, O'Nan \cite{ONan1972} noted that these configurations do not occur in the classical unital.  It was conjectured by Piper \cite{Piper1979} that this property characterises the classical unital within the class of all unitals.

In 2019, Feng and Li \cite{FengLi2019} showed the existence of O'Nan configurations in orthogonal BM unitals in $\PG(2,q^2)$, for both $q$ even and odd.  They also show the existence of O'Nan configurations in the Buekenhout-Tits unital.

We will consider a configuration called a \emph{Triple O'Nan configuration}, or Triple O'Nan, which is a set of six distinct lines pairwise intersecting in exactly 7 points of the unital, containing three distinct O'Nan configurations.

\begin{figure}[ht]
\centering
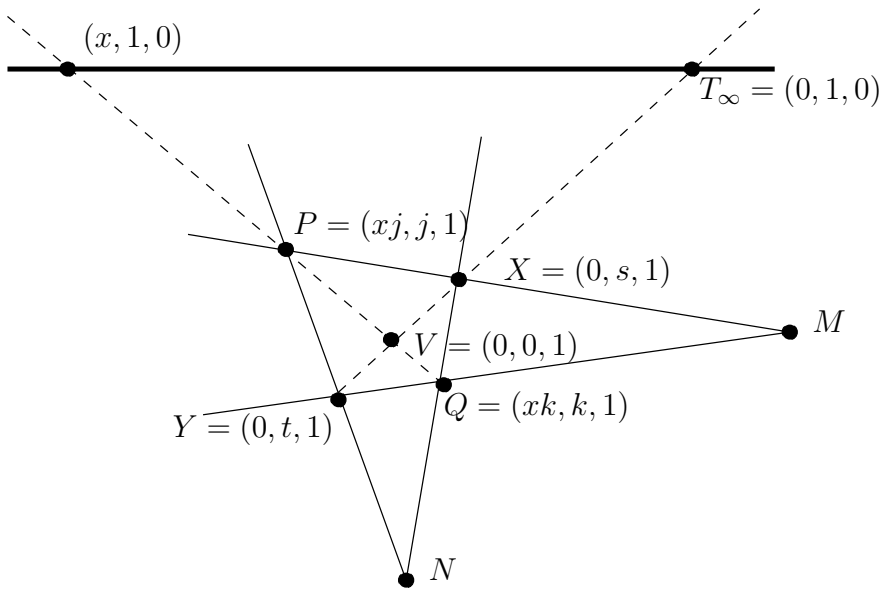
\caption{BM-special Triple O'Nan}\label{fig:bigonan}
\end{figure}

Note that an O'Nan configuration can also be considered as four points of the unital forming a quadrangle with at least two of the three opposite pairs of sides meeting in points of the unital. If all three opposite pairs of sides are points of the unital, then this is a Triple O'Nan of the unital. More formally:

\begin{defn}  In $\PG(2,q^2)$, $q$ even or odd, consider a BM unital $\U$ with special point $T_\infty$ at infinity. A \emph{Triple O'Nan} configuration is a set of six lines defined in the following way. Suppose there exists a quadrangle of four points in $\U$ such that, for the three pairs of lines partitioning the points of the quadrangle, each pair intersects in a point of the unital. These three pairs of lines forming six lines are a \emph{Triple O'Nan}.
\end{defn}

Note that, as the definition of an O'Nan configuration is in terms of lines, there are three quadrangles in the six points, each determining a seventh point, so there are three possibilities for a Triple O'Nan containing a given O'Nan configuration. Figure~\ref{fig:bigonan} shows dotted lines being one of the three possible pairs of lines extending the four solid lines of an O'Nan configuration to a Triple O'Nan configuration..

More specifically, if $\{P,Q,X,Y\}$ are four points of a unital lying on a quadrangle, with $V=PQ\cap XY,M=PX\cap QY,N=PY\cap QX$ also points of the unital, then the six lines $PQ,PX,PY,QX,QY,XY$ is a Triple O'Nan with associated points $P,Q,X,Y,V,M,N$ of the unital. The  three O'Nan configurations are
\[
\{PX,QY,PY,QX\},\
\{PX,QY,PQ,XY\} \quad\mbox{and}\quad
\{PY,QX,PQ,XY\}.
\]

In this paper, we show the existence of Triple O'Nan configurations in BM unitals of odd order, and make the following conjecture:

\begin{conjecture}\Label{conj:big}
  Every non-classical unital $\U$ in $\PG(2,q^2)$ contains a Triple O'Nan configuration if $q>5$.
\end{conjecture}

\subsection{Triple O'Nan configurations in Buekenhout-Metz unitals}
In Baker and Ebert \cite{BakerEbert1992}, it is shown that no O'Nan configuration contains the special (infinite) point $T_\infty=(0,1,0)$ of a Buekenhout-Metz unital \cite[Lemma 7.42]{thebook}. Hence it follows there is no Triple O'Nan whose associated points contains the special point $T_\infty$. However, it is possible that a line of the Triple O'Nan as a line of the underlying plane contains $T_\infty$.

\begin{defn} In $\PG(2,q)$, $q$ odd, consider a BM unital  $\U$ with special point $T_\infty$ at infinity.
  We say a Triple O'Nan configuration in a BM unital in $\PG(2,q^2)$ with special point $T_\infty$ is a \emph{BM-special Triple O'Nan configuration} if one of the six lines of the Triple O'Nan contains the special point $T_\infty$.
  Otherwise, we call it a  \emph{BM-ordinary Triple O'Nan configuration}.
\end{defn}

\subsection{Coordinates of the orthogonal BM unitals}\Label{sec:defn}

For $q$ odd, Baker and Ebert~\cite{BakerEbert1992} showed that a BM unital in $\PG(2,q^2)$ is projectively equivalent to the following set of points
\[
\Uab=\{(x,ax^2+bx^{q+1}+r,1)\bigm| r\in\Fq,\ x\in\Fqq\}\cup\{T_\infty=(0,1,0)\}.
\]
where the discriminant $d=(b-b^q)^2+4a^{q+1}$ is a non-square in $\Fq$.
Further $\Uab$ is classical if and only if $a=0$.

If $b=0$ then $a$ is necessarily non-square, and $\U_{a,0}$ is called the \emph{conic} unital, as there is a decomposition of the affine points of $\U_{a,0}$ into affine conics~\cite[Theorem 4.19, p77]{thebook}.

Further, a point $(x,y,1)$ of $\PG(2,q^2)$ belongs to $\Uab$
\begin{eqnarray*}
  &\Leftrightarrow& y=ax^2+bx^{q+1}+r\quad\mbox{for some $r\in\Fq$}\\
  &\Leftrightarrow& y-(ax^2+bx^{q+1}) \in\Fq\\
  &\Leftrightarrow& y-(ax^2+bx^{q+1}) =( y-(ax^2+bx^{q+1}))^q\\
  &\Leftrightarrow& ax^2-a^qx^{2q}+(b-b^q)x^{q+1}) =y-y^q
\end{eqnarray*}
Hence we have $(x,y,1)\in\Uab$ if and only if
\begin{equation}
  ax^2-a^qx^{2q}+(b-b^q)x^{q+1} =y-y^q.\Labele{eqn:unital-pt}
\end{equation}

%=============================================================================

\subsection{The automorphism group of the BM unital, odd $q$}\Label{sec:autosG}

The automorphism group $G$ of the BM unital $\Uab$ was described by Baker and Ebert~\cite{BakerEbert1992}. For $a\ne 0$, $G$ is shown to fix the point $T_\infty=(0,1,0)$ and is generated by the following three automorphisms:
\[
\begin{array}{rll}
  \phi_t\colon &\ \mkern-18mu(x,y,z)\mapsto(x,y+tz,z),& t\in\Fq\\
  \psi_\gamma\colon &\ \mkern-18mu(x,y,z)\mapsto(x+\gamma z,(2a\gamma-(b^q-b)\gamma^q)x+y+(a\gamma^2+b\gamma^{q+1})z,z),& \gamma\in\Fqq\\
  \mu_\delta\colon &\ \mkern-18mu(x,y,z)\mapsto(\delta x,\delta^2y,z),&
\end{array}
\]
where if $b\in\Fq$ then $\delta^2\in\Fq^*$ and if $b\not\in\Fq$ then $\delta\in\Fq^*$.

The group $K=\{\phi_t\bigm| t\in\Fq\}$ fixes each point on $\li$, and each line through $T_\infty=(0,1,0)$, {\it i.e.} is an elation subgroup of centre $T_\infty$ and axis $\li$.  So $K$ is transitive on the points of $\Uab$ on each fixed line through $T_\infty$.

The group $S=\{\psi_\gamma\phi_t\bigm| \gamma\in\Fqq,\ t\in\Fq\}$ has order $q^3=|\Uab\backslash \{T_\infty\}|$ and acts regularly on $\Uab\backslash \{T_\infty\}$.

For more detail see  \cite{BakerEbert1992}.

We will use the automorphism group $G$ of a non-classical BM-unital to find the form of a Triple O'Nan.
As $G$ fixes the unital and is transitive on the lines through $T_\infty$, and on the points through the lines through the special point $T_\infty$, any BM-special Triple O'Nan is equivalent to one that contains the line $[1,0,0]$ with affine unital points $(0,c,1)$ with $c\in\Fq$, and we can further assume the the Triple O'Nan configuration contains the point $V=(0,0,1)$, with two further points $X=(0,s,1)$ and $Y=(0,t,1)$ of the unital, with $s,t\in\Fq^*$ and $s\ne t$.

Exactly one of the points $V,X,Y$ has just two lines of the Triple O'Nan through it, and we {may} choose this point to be $V$. One line through $V$ is then $VXY$. The other line contains two points $P,Q$ of the Triple O'Nan configuration. Then the six lines of the Triple O'Nan are $PX,QY,PY,QX,VXY,VPQ$.

The line $VPQ$ has infinite {point} $(x,1,0)$ for some $x\in\Fqq^*$ (note that $[0,1,0]$ with infinite point $(1,0,0)$ is a tangent of the  unital, and contains only the point $(0,0,1)$ of the unital, and hence is not a line of the unital), and so $P$ and $Q$ have coordinates $(xj,j,1)$ and $(xk,k,1)$ respectively, for some $j,k\in\Fqq^*$, $j\ne k$ (note that $V,P,Q$ are distinct).  The four points $P,Q,X,Y$ are a quadrangle that generates the BM-special Triple O'Nan.

\subsection{Equivalence of BM unitals}\Label{sec:equiv}

Baker and Ebert \cite{BakerEbert1992} also investigated the equivalence classes of BM unitals under the automorphism group of the plane. They show that two BM unitals $\Uab$, $\U_{a',b'}$ are equivalent if and only if there exist $v\in\Fq^*$, $\gamma\in\Fqq^*$ and $u\in\Fq$ and $\tau$ a field automorphism of $\Fq$ with
\[
(a',b')=(a^\tau\gamma^2 v,\ b^\tau\gamma^{q+1}v+u).
\]
From~\cite[Corollary 4.15, p 74]{thebook}, if $q$ is an odd prime, then there are $\frac12(q+1)$ inequivalent BMunitals, one of which is the classical unital. Hence for $q=3,5$ there is the classical unital and the conic unital. For $q=5$ there is an additional unital $\Uab$ with $a$ being a non-zero square and $b\ne 0$.
\subsection{Main results of paper}

Feng and Li \cite{FengLi2019} have three  elegant constructions of O'Nan configurations: one for $q$ odd and two for $q$ even. In this paper we show that their  O'Nan configuration for $q$ odd cannot be extended to a Triple O'Nan configuration

\begin{theorem}\Label{thm:FLodd}
  The $q$ odd O'Nan construction of Feng and Li~\cite{FengLi2019} cannot be extended {to} a Triple O'Nan.
\end{theorem}

We investigate BM-special Triple O'Nans and show that:
\begin{theorem}\Label{thm:bigonanoddnum}
  The number of BM-special Triple O'Nan configurations on a non-classical orthogonal BM unital $\Uab$ in $\PG(2,q^2)$ for $q$ odd depends only on the value of $q$ and whether $a$ is a square or not.
\end{theorem}

\begin{theorem}\Label{thm:bigonanoddspecial}
For odd $q$, each non-classical orthogonal BM unital $\Uab$ contains a BM-special Triple O'Nan configuration, with the exception for $q=3,5$ {for the conic unitals}.
\end{theorem}

Finally, we investigate further properties of these Triple O'Nans.

%=============================================================================

\section{The O'Nan construction of Feng and Li}
We now describe the Feng and Li construction~\cite[Section 3.1]{FengLi2019}  for $q$ odd.
In $\PG(2,q^2)$, $q$ odd, let $\sigma=\mu_{-1}$ (see Section~\ref{sec:autosG}), so $\sigma\colon (x,y,z)\mapsto (-x,y,z)$. Let $P=(0,0,1)$ and let $\ell_1=[1,1,0]$ and $\ell_{-1}=[1,-1,0]$ be two lines through $P$.  Let $R=(0,r,1)\in\Uab$ for some $r\in\Fq^*$  and let $\ell'$, $\ell''=\sigma(\ell')$ be two lines through $R$. Then define
\begin{eqnarray}
  P_{\lambda_1}&=&\ell_1\cap\ell'=(x_{\lambda_1},-x_{\lambda_1},1),\quad
  x_{\lambda_1}=\frac{a^q+\lambda_1-b}{a^{q+1}-(\lambda_1-b)^{q+1}}\Labele{eqn:Pl1}\\
   P_{\lambda_2}&=&\ell_1\cap\ell''=(x_{\lambda_2},-x_{\lambda_2},1,)\quad
  x_{\lambda_2}=\frac{a^q+\lambda_2-b}{a^{q+1}-(\lambda_2-b)^{q+1}}\nonumber
\end{eqnarray}
and
\[
r=-\frac{2x_{\lambda_1}x_{\lambda_2}}{x_{\lambda_1}+x_{\lambda_2}}.
\]
Feng and Li show that there eixsts $\lambda_1,\lambda_2\in\Fq$ with $\lambda_1\ne \lambda_2$ such that the four lines
\[
\ell_1=P_{\lambda_1}P_{\lambda_2},\ \ell_{-1}=P_{\lambda_1}^\sigma P_{\lambda_2}^\sigma ,\
\ell'=P_{\lambda_1}P_{\lambda_2}^\sigma ,\ \ell''=P_{\lambda_1}^\sigma P_{\lambda_2}.
\]
form an O'Nan configuration.

We now prove Theorem~\ref{thm:FLodd}.
\begin{proof}

  As discussed earlier (Section~\ref{sec:intro}), to extend the O'Nan to a Triple O'Nan, we calculate the final diagonal points of the three quadrangles of the configuration and check whether any of these belong to the unital. Let $R=(0,r,1)$. The  three final diagonal points are
  \begin{eqnarray*}
    &&I_0=P_{\lambda_1}P_{\lambda_1}^\sigma \cap P_{\lambda_2}P_{\lambda_2}^\sigma=(1,0,0)\\
    &&I_1=PR\cap P_{\lambda_1}P_{\lambda_1}^\sigma=(0,-x_{\lambda_1},1)\\
    &&I_2=PR\cap P_{\lambda_2}P_{\lambda_2}^\sigma=(0,-x_{\lambda_2},1)
  \end{eqnarray*}
  We note that $I_0\in \li=[0,0,1]$ is not a point of $\Uab$, as the only point of $\Uab$ on $\li$ is $T_\infty=(0,1,0)$.
  As $I_1,I_2,P,R,(0,1,0)$ are collinear, and $PR=[1,0,0]$ and the points of $\Uab$ on $PR$ are $(0,c,1)$ with $c\in\Fq\cup\{\infty\}$, it follows that $I_1$ and $I_2$ are points of $\Uab$ if and only if $x_{\lambda_1}$ and $x_{\lambda_2}$ are  elements of $\Fq$.

  Now  $x_{\lambda_1}\in\Fq$  if and only if $x_{\lambda_1}^q=x_{\lambda_1}$, which holds by (\ref{eqn:Pl1}) if and only if $a-b\in\Fq$. Write $a=a_0+a_1\e$, $b=b_0+b_1\e$. If $a-b\in\Fq$ then $a_1=b_1$ {and the} discriminant for $\Uab$ is
  \[
d=(b-b^q)^2+4aa^q=(4b_1^2)+4(a_0^2-wa_1^2)=4a_0^2
  \]
  which we see is a square in $\Fq$, contradicting the definition of $\Uab$ (Section~\ref{sec:defn}). Thus $I_1$ is never a point of $\Uab$. The argument for  $x_{\lambda_2}$ is the same. As none of the three possible points belong to $\Uab$, it follows that the O'Nan is not contained in a Triple O'Nan.
\end{proof}

\section{BM-special Triple O'Nan configurations in BM unitals for $q$ odd}\Label{sec:typea}

\subsection{Construction of BM-special Triple O'Nans for $q$ odd}\Label{sec:TypeA}

We now give a construction for a putative BM-special Triple O'Nan and examine under what conditions it is a Triple O'Nan.

For $q\ge 3$ consider $\Uab$ in  $\PG(2,q^2)$. By Section~\ref{sec:autosG}, without loss of generality, let $P=(xj,j,1)$ and $Q=(xk,k,1)$ with $x,j,k\in\Fqq^*$, $X=(0,s,1)$ and $Y=(0,t,1)$ with $s,t\in\Fq^*$. Let $V=(0,0,1)$. Let $M=PX\cap QY$ and $N=PY\cap QX$. Then the six lines (see Figure~\ref{fig:bigonan})
\[
PXM,YQM,PYN,XQN,PVQ,YVX
\]
form a BM-special Triple O'Nan if and only if, the seven points $P,Q,X,Y,V,M,N$ belong to $\Uab$, if and only if,  $P,Q,M,N$ belong to the $\Uab$ (as by definition the points $V,M,N$ belong to $\Uab$).

We calculate
\begin{eqnarray*}
  PX=[j-s,-xj,sxj],\quad &&QX=[k-s,-xk,sxk]\\
  PY=[j-t,-xj,txj],\quad &&QY=[k-t,-xk,txk]
\end{eqnarray*}
and so
\begin{eqnarray*}
  M&=&\left(
  \frac{jkx(s-t)}{sk-tj},\
  \frac{jk(s-t)+st(k-j)}{sk-tj},\
    1
    \right)\\
    N&=&\left(
  \frac{jkx(t-s)}{tk-sj},\
  \frac{jk(t-s)+st(k-j)}{tk-sj},\
    1
    \right)
\end{eqnarray*}
As $j,k\in\Fqq^*$, we can write $h=j/k\in\Fqq^*$.
Let
\begin{eqnarray}
W=\frac{\rd(s-t)}{s-t\rd},\quad U=\frac{st(1-\rd)}{s-t\rd}\Labele{eqn:WUdef}
\end{eqnarray}
and swapping $s$ and $t$ let
\begin{eqnarray}
V=\frac{\rd(t-s)}{t-s\rd},\quad Z=\frac{st(1-\rd)}{t-s\rd}.\Labele{eqn:VZdef}
\end{eqnarray}
An easy calculation {shows} that
\[
M=\left(kxW,kW+U,1\right),\quad
N=\left(kxV,kV+Z,1\right).
\]
By definition $V,X,Y$ belong to $\Uab$. Using (\ref{eqn:unital-pt}), the conditions that $P,Q,M,N$ belong to $\Uab$ are:
\begin{eqnarray}
  a(xk)^2-a^q(xk)^{2q}+(b-b^q)(xk)^{q+1}&=&k-k^q \Labele{eqn:P}\\
  a(xk)^2h-a^q(xk)^{2q}h^q+(b-b^q)(xk)^{q+1}\rd^{q+1}&=&\rd k-(\rd k)^q\Labele{eqn:Q}\\
  a(xk)^2W^2-a^q(xk)^{2q}W^{2q}+(b-b^q)(xk)^{q+1}W^{q+1}&=&(kW+U)-(kW+U)^q \Labele{eqn:M}\\
  a(xk)^2V^2-a^q(xk)^{2q}V^{2q}+(b-b^q)(xk)^{q+1}V^{q+1}&=&(kV+Z)-(kV+Z)^q \Labele{eqn:N}
 \end{eqnarray}
We call these equations the \emph{BM-special Triple O'Nan equations}.

%\newpage
\subsection{Existence Conditions for BM-special Triple O'Nans, for $q$ odd}

Using the notation for  BM-special Triple O'Nans as discussed in Section~\ref{sec:TypeA}, we aim to show that, for a given value of $q$, the number of  BM-special Triple O'Nans with points $P,Q,V,X,Y,M,N$, for given $s,t,h$,  depends only on whether $a$ is square or non-square in $\Fqq$.

Fix $q$. The parameter $a$ of $\Uab$ is either a square or non-square in $\Fqq$. By Section~\ref{sec:equiv}, if $a$ is a square then we can assume that  $a$  is a fixed element $a_s\in\Fq$. If $a$ is a non-square, then we can assume that $a$ is a fixed non-square element $a_n\in\Fqq\backslash \Fq$. Further, we can assume that  $b$ is of the form $b_1e$ with $b_1\in\Fq$ where $e$ is a fixed element of $\Fqq\backslash \Fq$.

\begin{lemma} \Label{lem:sq}
Suppose there exists a BM-special Triple O'Nan  in the BM unital {${\mathcal U}_{a,b_1e}$} with $b_1\in\Fq$ with parameters  $(a,b_1e,x,k,\rd,s,t)$ with $j=k\rd$.  Let
\begin{equation}
  \Delta={\Delta_{{\mathcal U}_{a,b_1e}}}=ax^2k^2,\quad \Theta=(xk)^{q+1}=\sqrt{ (\textstyle\frac\Delta a)^{q+1}}.\Labele{def:deltatheta}
\end{equation}

Suppose {${\mathcal U}_{a,b_2e}$} is another unital with $b_2\in\Fq$. Then there exists $y,m\in\Fqq$ and an  BM-special Triple O'Nan  with parameters $(a,b_2e,y,m,\rd,s,t)$ in {${\mathcal U}_{a,b_2e}$}.
\end{lemma}

\begin{corollary}\Label{cor:oddqnum}
  For a given value of $q$ odd, the number of BM-special Triple O'Nans in $\Uab$ is only dependent on whether $a$ is a square or not in $\Fqq$.
\end{corollary}

\begin{proof} (of Lemma~\ref{lem:sq})
We will rewrite the equations so they are easier to manipulate. Firstly we need some notation. For any element or expression $Y$ over $\Fqq$, write
\[
\boxed{Y} = Y-Y^q.
\]
As we have a BM-special Triple O'Nan in {${\mathcal U}_{a,b_1e}$}, consider the \emph{BM-special Triple O'Nan equations}, with $\Delta,\Theta$ defined in (\ref{def:deltatheta}), $W,U$ defined in (\ref{eqn:WUdef}) and $V,Z$ defined in (\ref{eqn:VZdef}).
\begin{eqnarray}
  \boxed{\Delta}+2b_1e\Theta&=&\boxed{k}\Labele{eqn:P1}\\
  \boxed{\Delta h}+2b_1e\Theta\rd^{q+1}&=&\boxed{k\rd}      \Labele{eqn:Q1}\\
  \boxed{\Delta W^2}+2b_1e\Theta W^{q+1}&=&\boxed{kW}+\boxed{U}\ \Labele{eqn:M1}\\
  \boxed{\Delta V^2}+2b_1e\Theta V^{q+1}&=&\boxed{kV}+\boxed{Z}  \Labele{eqn:N1}
\end{eqnarray}

Now consider {${\mathcal U}_{a,b_2e}$} with $b_2\in\Fq$. We will show there exists $y,m$ with $ax^2k^2=ay^2m^2=\Delta$, and so $(xk)^{q+1}=(ym)^{q+1}=\Theta$. Note that $W,V,U,Z$ only involve $\rd,s,t$ which is the same for both configurations. Hence $(a,b_2e,y,m,\rd,s,t)$ is a BM-special Triple  O'Nan if the following four equations hold
\begin{eqnarray}
  \boxed{\Delta}+2b_2e\Theta&=&\boxed m\Labele{eqn:P2}\\
  \boxed{\Delta h}+2b_2e\Theta\rd^{q+1}&=&\boxed{m\rd}     \Labele{eqn:Q2}\\
  \boxed{\Delta W^2}+2b_2e\Theta W^{q+1}&=&\boxed{mW}+\boxed{U}\ \Labele{eqn:M2}\\
  \boxed{\Delta V^2}+2b_2e\Theta V^{q+1}&=&\boxed{mV}+\boxed{Z}  \Labele{eqn:N2}
\end{eqnarray}

We will use (\ref{eqn:P2}) and (\ref{eqn:Q2}) to define $y,m$ with $ay^2m^2=\Delta$. Firstly, we will define $m$ satisfying (\ref{eqn:P2}) and (\ref{eqn:Q2}) and then we will let $y^2=\frac{\Delta}{am^2}$. Now consider  (\ref{eqn:P1}) minus (\ref{eqn:P2}) and similarly  (\ref{eqn:Q1}) minus (\ref{eqn:Q2}) to get
\begin{eqnarray}
  2(b_1-b_2)e\Theta&=&k-k^q-(m-m^q)\Labele{eqn:W1}\\
  2(b_1-b_2)e\Theta\rd^{q+1}&=&k\rd-(k\rd)^q-( m\rd-(m\rd)^q) \nonumber\\
  &=&
  \rd(k-m)-\rd^q(k-m)^q\Labele{eqn:W2}
\end{eqnarray}
If we write $m=m_0+m_1e$ then we see from (\ref{eqn:W1}) that we can get an expression for $m_1$ in terms of $k_1$ and  $2(b_1-b_2)e\Theta$.  Similarly from (\ref{eqn:W2}) this is linear in $m_0,m_1,k_0,k_1$ and we can find an expression for $m_0$ in terms of $l_0,k_1$ and $\Theta$. This defines $m$ and hence $y$ as discussed.

It remains to show that (\ref{eqn:M2}) and (\ref{eqn:N2}) hold. For  (\ref{eqn:M2}), given that (\ref{eqn:M1}) holds, this is equivalent to showing that (\ref{eqn:M1}) minus (\ref{eqn:M2}) holds. That is, show that
\begin{eqnarray}
  2(b_1-b_2)e\Theta W^{q+1}&=&(k-m)W-(k-m)^qW^q\Labele{eqn:P4a}
\end{eqnarray}
and similarly
\begin{eqnarray}
  2(b_1-b_2)e\Theta V^{q+1}&=&(k-m)V-(k-m)^qV^q\Labele{eqn:P4b}
\end{eqnarray}
We have
\[
W=\frac{\rd(s-t)}{s-t\rd},\quad
W^q=\frac{\rd^q(s-t)}{s-t\rd^q},\quad
W^{q+1}=\frac{\rd^{q+1}(s-t)^2}{(s-t\rd)(s-t\rd^q)},
\]
so
\begin{eqnarray*}
  &&\mbox{RHS (\ref{eqn:P4a})}\\
  &=&(k-m)\frac{\rd(s-t)}{s-t\rd}-(k-m)^q \frac{\rd^q(s-t)}{s-t\rd^q}\\
  &=&\frac{s-t}{(s-t\rd)(s-t\rd^q)}\left(
\rd(k-m)(s-t\rd^q)-\rd^q(k-m)^q(s-t\rd)
\right)\\
&=&\frac{s-t}{(s-t\rd)(s-t\rd^q)}\Bigl(
s(\rd(k-m)-\rd^q(k-m)^q)\\
&&\quad\quad\quad\quad -t\rd^{q+1}((k-m)-(k-m)^q)
\Bigr)\\
&=&\frac{s-t}{(s-t\rd)(s-t\rd^q)}\Bigl(
  s2(b_1-b_2)e\Theta\rd^{q+1}-t\rd^{q+1} 2(b_1-b_2)e\Theta
  \Bigr)\\
  &&\quad\quad\quad \mbox{by (\ref{eqn:W2}) and (\ref{eqn:W1})}\\
&=&\frac{\rd^{q+1}(s-t)^2}{(s-t\rd)(s-t\rd^q)}2(b_1-b_2)e\Theta\\
&=&W^{q+1}2(b_1-b_2)e\Theta\\
&=&\mbox{LHS (\ref{eqn:P4a})}
\end{eqnarray*}
and the proof of (\ref{eqn:P4b}) is similar. This shows that the configuration is an  BM-special Triple O'Nan as required.
\end{proof}

\subsection{Non-existence of BM-special Triple O'Nans for $q=3,5$}

The unitals for $q=3,5$ were discussed in Section~\ref{sec:equiv}, and for $a$ non-square, there is only {one} class of equivalent unitals, being the conic unital. A search using Magma~\cite{magma} shows:
\begin{lemma}\Label{lem:qeq35}
  For $q=3$ and $5$, if $a$ is non-square, the {conic} unital  does not contain any BM-special Triple O'Nans.
\end{lemma}

\subsection{Existence of BM-special O'Nans for $q>5$ odd, $a$ non-square: the Conic Unital}\Label{sec:conicA}

Consider a BM-special Triple O'Nan, using the notation as above. Recall that for the conic unital we may assume that $b=0$. Hence the conditions that $P$ and $Q$ belong to the unital are
\begin{eqnarray*}
  a(xj)^2-j\in\Fq\quad\mbox{and}\quad  a(xk)^2-k\in\Fq.
\end{eqnarray*}

\begin{lemma}\Label{lem:suff} Given odd $q$, in the conic unital $\Uab$ with $b=0$, the following conditions are sufficient for there to be an  BM-special Triple O'Nan with parameters $(a,0,x,k,\rd,s,t)$
  \begin{enumerate}
  \item $s,t\in \Fq^*$, $s^2\ne t^2$, $s^2+t^2\ne 0$ {\it i.e.}  $u^4\ne 1$ where $u=s/t$.
  \item $h^2\in\Fq$, $h^2$ non-square in $\Fq$, $h^2\ne -1$.
  \item $k= c(1+\frac1{h})$ where $c=\frac{s^2+t^2}{2(s+t)}$.
  \item
    \begin{equation}
    2a(xk)^2=-\left(1+\frac1{h^2}\right)\frac{st}{s+t} +\frac1{h}\frac{s^2+t^2}{s+t}\Labele{eqn:stcond}
    \end{equation}
  \end{enumerate}
Hence a BM-special Triple O'Nans with parameters  $(a,0,x,k,\rd,s,t)$  exists for suitable choices of $s,t,h$ if and only if the expression for $2a(xk)^2$ above is a non-square in $\Fqq$.
\end{lemma}

\begin{proof}
In part 1, $s^2\ne t^2$ is equivalent to $u^2\ne 1$, and $s^2+t^2\ne 0$ is equivalent to $u^2\ne -1$, hence this is equivalent to $u^4\ne 1$. Note that such a $u$ exists only if $q>5$.

  We first note, for the final conclusion, that, for conic unitals, $a$ is a non-square in $\Fqq$, and so the equation above has a solution for $x$ whenever the expression for $2a(xk)^2$ is a non-square in $\Fqq$.

  Suppose we satisfy the conditions, so $x,k,h,s,t$ exist satisfying the conditions. Recall that as $h^2\in\Fq$, and $h^2$ is a non-square in $\Fq$, it follows that $h^q=-h$. Note that multiplying an element of $\Fqq$ by a square in $\Fqq$ does not change whether it is a square or not in $\Fqq$.

  We show that $P,Q,X,Y,M,N$ all belong to $\Uab$ {we} already know $V\in \Uab$).

Recall that $Q=(xk,k,1)$ so we have
\begin{eqnarray*}
  &&a(xk)^2-a^2(xk)^{2q}-k+k^q\\
  &=&\frac12\left(\frac{2c}{h}-\left(1+\frac1{h^2}\right)\frac{st}{s^2+t^2}\right)-
  \frac12\left(\frac{2c}{h}-\left(1+\frac1{h^2}\right)\frac{st}{s^2+t^2}\right)^q-\frac{2c}{h}\\
  &=&\frac12\left(\frac{2c}{h}-\left(1+\frac1{h^2}\right)\frac{st}{s^2+t^2}\right)-
  \frac12\left(\frac{2c}{(-h)}-\left(1+\frac1{h^2}\right)\frac{st}{s^2+t^2}\right)-\frac{2c}{h}\\
  &=&\frac12\left(\frac{4c}{h}\right)-\frac{2c}{h}\\
  &=&0
\end{eqnarray*}
so by (\ref{eqn:unital-pt}) this means $Q\in \Uab$. Similarly using $j=h k$ we can show $P\in \Uab$.

We now show that $M\in \Uab$ and $N\in \Uab$.

Firstly we note that $sk-tj=sk-th k=k(s-th)$ so we will multiply (\ref{eqn:unital-pt}) by $(s-th)^{2(q+1)}$ to simplify.
From the definition of $M$ {(Section~\ref{sec:TypeA})} consider the {LHS of (\ref{eqn:unital-pt})}
\begin{eqnarray}
  &&\left[a\left(\frac{{h^2} kkx(s-t)^2}{k(s-th)}\right)^2-a^q\left(\frac{h kkx(s-t)^2}{k(s-th)}\right)^{2q}\right](s-th)^{2(q+1)}\nonumber\\
  &=&{h^2}(s-t)^2\left[ak^2x^2(s+th)^2-a^qk^{2q}x^{2q}(s-th)^2\right]\nonumber\\\
  &=&{h^2}(s-t)^2\left[(s^2+{h^2}t^2)[ak^2x^2-a^qk^{2q}x^{2q}]+2sth(ak^2x^2+a^qk^{2q}x^{2q})\right]\nonumber\\\
  &=&{h^2}(s-t)^2\left[(s^2+{h^2}t^2)[k-k^q]+2sth(ak^2x^2+a^qk^{2q}x^{2q})\right]\nonumber\\\
  &=&{h^2}(s-t)^2\left[(s^2+{h^2}t^2)\frac{2c}{h}+2sth\left(-\left(1+\frac1h\right)\left(\frac{st}{s+t}\right)\right)\right]\nonumber\\\
  &=&{h^2}(s-t)^2\left[(s^2+{h^2}t^2)\left(\frac{s^2+t^2}{s+t}\right)+2sth({h^2}+1)\left(\frac{st}{s+t}\right)\right]\nonumber\\\
  &=&\frac{(s-t)^2}{s+t}h\left[(s^2+{h^2}t^2)(s^2+t^2)-2st({h^2}+1)\right]\nonumber\\\
   &=&(s-t)^3h(s^2-{h^2}t^2).\Labele{eqn:lhs1}
\end{eqnarray}
Now consider the {RHS of (\ref{eqn:unital-pt})}
\begin{eqnarray*}
  &=&\left[\left(\frac{h k(s-t)^2}{(s-th)}\right)-\left(\frac{h k(s-t)^2}{(s-th)}\right)^q+\left(\frac{st(1-h)}{s-th}\right)-\left(\frac{st(1-h)}{s-th}\right)^q\right](s-th)^{2(q+1)}\\
  &=&(s-th)^{q+1}\Bigl[(s-t)h\left(k(s+th)+k^q(s-th)\right)\\
    &&\quad\quad    +\ st\left((1-h)(s+th)+(1+h)(s-th)\right)\Bigr]\\
  &=&(s-th)^{q+1}\left[(s-t)h\left(s\cdot 2c+th\cdot \frac{2c}{h}\right)+st\left(2h(t-s)\right)\right]\\
  &=&(s^2-h^2t^2)(s-t)h(s-t)^2
\end{eqnarray*}
which is equal to (\ref{eqn:lhs1}), hence (\ref{eqn:unital-pt}) holds, proving that $M$ is a point of the unital. The argument showing that $N$ belongs to the unital is similar.
\end{proof}

Now it remains to find $s,t,h$ satisfying the requirements discussed above.

 \begin{lemma}\Label{lem:aoddq34}
   For odd $q$ with
   $q\ge 7$ there exists BM-special Triple O'Nan configurations for conic unitals.
  \end{lemma}
 \begin{proof} As discussed in the proof of Lemma~\ref{lem:suff}, we need $q>5$.

We consider conditions that make (\ref{eqn:stcond}) a non-square in $\Fqq$.

We simplify (\ref{eqn:stcond})
\begin{eqnarray*}
  &&-\left(1+\frac1{h^2}\right)\frac{st}{s+t} +\frac1{h}\frac{s^2+t^2}{s+t}\\
  &=&-\frac1{h^2(s+t)}\left((h^2+1)st-h(s^2+t^2)\right)\\
  &=&-\frac1{h^2(s+t)}(t-hs)(s-ht).
\end{eqnarray*}
As $h^2,s,t\in\Fq$, this expression is a non-square in $\Fqq$ if and only if $(t-hs)(s-ht)$ is a non-square in $\Fqq$, if and only if the following is a non-square in $\Fq$:
\[
\left((t-hs)(s-ht)\right)^{q+1}=(t^2-h^2s^2)(s^2-h^2t^2)=\frac1{s^2t^2}\left(-\frac{h^2s^2}{t^2}+1\right)
\left(-\frac{h^2t^2}{s^2}+1\right)
\]
Write $u=s/t\in\Fq$, and as $s^2t^2$ is a square in $\Fq$, this {expression} is a non-square in $\Fq$ if and only if
\begin{equation}
(-h^2/u^2+1)(-h^2u^2+1)=(X+1)(Y+1)\Labele{eqn:XY}
\end{equation}
is a non-square in $\Fq$, where $X=-h^2/u^2,Y=-h^2u^2$.

We will use the results from Baumert \cite[p119 onwards]{baumert}. We treat the cases $q \equiv 1 \pmod 4$ and $q \equiv 3 \pmod 4$ separately.

\subsubsection{BM-special Triple O'Nans for $q>5$ odd, $a$ non-square, $q\equiv 1\pmod 4$}
Suppose firstly that $q\equiv 1\pmod 4$, so $q=4m+1=2f+1,f=2m$. We note that $-1$ is a square.The cases when $m$ is even or odd have different but similar treatments.

 We make use of  cyclotomic numbers (see Dickson or Baumert (following theorem 5.18) for more detail). Put
\[
R_0=\{w^{4t}\mid  t \in \mathbb Z \}.
\]
Then $R_0$ is a subgroup of $\Fq^*$ of order $\frac{q-1}{4}$. Put
\[
R_i=w^iR_0=\{ w^{4t+i} \bigm| t \in  \mathbb Z,\ i=1,2,3\}
\]
For $0 \le i,j \le 3$ define the following \emph{cyclotomic numbers}
\[
(i,j)_4=|\{ (x,x+1)\mid x \in \Fq^*,\ x \in R_i,\ x+1 \in R_j \}|.
\]

Baumert gives the following equations~\cite[fourth paragraph following thm 5.16]{baumert} that the cyclotomic numbers satisfy, where, here and throughout, the subscripts are treated as (arbitrary) residues modulo $4$ and $\delta_{ij}$ is $1$ if $i=j$ and is $0$ otherwise. The following equations are from~\cite[5.21-3]{baumert}:

When $q\equiv 1 \pmod 8$:

\begin{equation}
  (i,j)_4=(-i,j-i)_4=(j,i)_4,\quad  \sum_{j=0}^3 (i,j)_4=\frac{q-1}{4}-\delta_{0i}.\Labele{eqn:q1-8}
\end{equation}

When $q\equiv 5 \pmod 8$:

\begin{equation}
  (i,j)_4=(-i,j-i)_4=(j+2,i+2)_4,\quad   \sum_{j=0}^3 (i,j)_4=\frac{q-1}{4}-\delta_{2i}.\Labele{eqn:q5-8}
\end{equation}

 These equations do not determine the cyclotomic numbers, as their values may depend on the choice of generator $w$ of the multiplicative group $\Fq^*$. However, they give the following solutions with two indeterminants $\ell_1,\ell_2$. The equality conditions give {a} grouping into five groups {of equal cyclotomic numbers}, the summations  (\ref{eqn:q1-8}) allows three of {these values} to be written in terms of the other two.

 When $q\equiv 1 \pmod 8$, using (\ref{eqn:q1-8}):
 \begin{equation}
 \begin{array}{ccccccccl}
   (0,0)_4&+&(0,1)_4&+&(0,2)_4&+&(0,3)_4&=&\frac{q-1}4-1\\
   (1,0)_4&+&(1,1)_4&+&(1,2)_4&+&(1,3)_4&=&\frac{q-1}4\\
   (2,0)_4&+&(2,1)_4&+&(2,2)_4&+&(2,3)_4&=&\frac{q-1}4\\
   (3,0)_4&+&(3,1)_4&+&(3,2)_4&+&(3,3)_4&=&\frac{q-1}4
\end{array} \Labele{eqn:big}
 \end{equation}
 If we start with $(1,2)_4$ using (\ref{eqn:q1-8}) we get $(1,2)_4=(-1,2-1)_4=(3,1)$. Also $(2,1)_4=(1,2)_4$, and so on.  We choose $\ell_1=(1,2)_4$, $\ell_2=(0,3)_4$  to be indeterminants  and get
\begin{eqnarray*}
  &&(1,2)_4=(1,3)_4=(2,1)_4=(2,3)_4=(3,1)_4=(3,2)_4=\ell_1\\
  &&(0,3)_4=(3,0)_4=(1,1)_4=\ell_2\\
  &&(0,2)_4=(2,0)_4=(2,2)_4=\textstyle\frac{q-1}{8}-\ell_1\\
  &&(0,1)_4=(1,0)_4=(3,3)_4=\textstyle\frac{q-1}{4}-2\ell_1-\ell_2\\
  &&(0,0)_4=3\ell_1-\textstyle\frac{q+7}{8}
\end{eqnarray*}
where the third, fourth and fifth equations use the third, second and first equations in (\ref{eqn:big}).

When $q \equiv 5 \pmod 8$, similarly,  we choose  $\ell_1=(1,0)_4$, $\ell_2 = (0,3)_4$ to be indeterminants and get
\begin{eqnarray*}
&&(1,0)_4=(1,1)_4=(2,1)_4=(2,3)_4=(3,0)_4=(3,3)_4=\ell_1\\
&&(0,1)_4=(1,3)_4=(3,2)_4=\ell_2\\
&&(0,3)_4=(1,2)_4=(3,1)_4=\textstyle\frac{q-1}{4}-2\ell_1-\ell_2\\
&&(0,0)_4=(2,0)_4=(2,2)_4=\textstyle\frac{q-5}{8}-\ell_1\\
&&(0,2)_4=3\ell_1-\textstyle\frac{q-5}{8}
\end{eqnarray*}

Recall that we require $XY$ to be a square and $(X+1)(Y+1)$ to be a non-square. Further we need $X/Y$ to be a fourth power, so we need $X,Y$ both in $R_1$ or both in $R_3$. Thus we need
\begin{enumerate}
\item $X,Y\in R_1$ and $X+1\in R_0\cup R_2$ and $Y+1\in R_1\cup R_3$, or
\item $X,Y\in R_3$ and $X+1\in R_1\cup R_3$ and $Y+1\in R_0\cup R_2$.
\end{enumerate}
We count the number $n$ of suitable values of $X$ and $Y$.
The number of $X$ with $X\in R_1$ and $X+1\in R_0\cup R_2$ is $(1,0)_4+(1,2)_4$ and the number of $Y$ with $Y\in R_1$ and $Y+1\in R_1\cup R_3$ is $(1,1)_4+(1,3)_4$  and similarly for the second part. Note that this automatically excludes $X=Y$. We get
\begin{eqnarray*}
  n&=&((1,0)_4+(1,2)_4)((1,1)_4+(1,3)_4)+((3,0)_4+(3,2)_4)((3,1)_4+(3,3)_4)\\
  &=&(\textstyle\frac{q-1}{4}-2\ell_1-\ell_2+\ell_1)(\ell_2+\ell_1)+(\ell_2+\ell_1)(\ell_1+\textstyle\frac{q-1}{4}-2\ell_1-\ell_2)\\
  &=&2(\textstyle\frac{q-1}{4}-\ell_1-\ell_2)(\ell_1+\ell_2)
\end{eqnarray*}
for  $q \equiv 1 \pmod 8$. The equation is the same for  $q \equiv 5 \pmod 8$. We will show $n>0$.

Firstly, for $q\equiv 1\pmod 8$, we have $q\ge 9$. From $(0,0)_4\ge 0$ we see $3\ell_1-\frac{q+7}{8}\ge 0$ and as $q\ge 9$ we have $\ell_1>0$ and so $\ell_1+\ell_2>0$.  Using $\ell_1>0$ we have  $\frac{q-1}{4}-\ell_1-\ell_2> \frac{q-1}{4}-2\ell_1-\ell_2=(0,1)_4\ge 0$ and so $n>0$ as required.

Secondly, for $q\equiv 5\pmod 8$, we assume that $q>5$. From $(0,2)_4\ge 0$ we get $3\ell_1\ge \frac{q-5}{8}>0$ so $\ell_1>0$ and hence $\ell_1+\ell_2>0$. From  $(0,3)_4\ge 0$ we get $\frac{q-1}{4}-2\ell_1-\ell_2\ge 0$ and so  $\frac{q-1}{4}-\ell_1-\ell_2>\frac{q-1}{4}-2\ell_1-\ell_2 \ge 0$ and so $n>0$.

We now show that such values for $X$ and $Y$ which make (\ref{eqn:stcond}) a non-square in $\Fqq$ give the required solution, {\it i.e.} the conditions of Lemma~\ref{lem:suff} are satisfied. Suppose firstly, that $X,Y\in R_1$, so we can write $X=wx^4$ and $Y=wy^4$ for some $x,y\in\Fq$. So from $X=-h^2u^2$ and $Y=-h^2/u^2$ we get $XY=(h^2)^2$ and $XY=(wx^2y^2)^2$ so we have $h^2=wx^2y^2$, a non square as required. As $-1$ is a square, we have $h^2\ne -1$, as required. In a similar way we get $u^4=X/Y=(wx^4)/(wy^4)=(\frac xy)^4$ so we have $u=\frac xy$. Note that $u\ne 0$ and $u^4=X/Y$ and as noted earlier, as $X\ne Y$ we have $u^4\ne 1$ as required. The second case where  $X,Y\in R_3$ is similar.

Thus we have proved the following lemma for the subcase $q \equiv 1 \pmod 4$.

\begin{lemma}\Label{lem:aoddq14}
    For odd $q$, with $q\equiv 1 \pmod 4$, $q >5$, there exists {Triple} O'Nan configurations for conic unitals.
  \end{lemma}

\subsubsection{BM-special Triple O'Nans for $q>3$ odd, $a$ non-square, $q\equiv 3\pmod 4$}

In a similar manner to the previous section, let $R_0=\{w^{2t}\bigm| t\in\Z\}$ be the set of squares of $\Fq^*$ (type 0 elements) and let $R_1=wR_0$ be the set of non-squares of $\Fq^*$ (type 1 elements). For $\ell,m\in\{0,1\}$, the notation $(\ell,m)_2$ is the number of elements $x,x+1\in\Fq^*$ such that $x$ is of type $\ell$ and $x+1$ is of type $m$.  There are $q-2$ such {numbers of} elements. For example, $(0,0)_2$ is the number of elements $x,x+1$ in $\F^*$ for which both $x$ and $x+1$ are squares in $\Fq$. Baumert \cite[p124]{baumert} presents equations relating the values of $(\ell,m)_2$.

We  consider $q\equiv 3\pmod 4$, so $q=4m+3=2f+1,f=2m+1$. In this case $-1$ is a non-square in $\Fq$. Then \cite[p121, 5.22]{baumert} gives
\begin{equation}
(1,0)_2=(1,1)_2=(0,0)_2=m=\frac{q-3}4,\quad (0,1)_2=m+1=\frac{q+1}4.\Labele{eqn:B2}
\end{equation}

We now show how to choose $u,h^2\in\Fq$ with the required properties.

We want to choose $X,Y$ to be of types $(0,0)_2,(0,1)_2$ (respectively) or  $(0,1)_2,(0,0)_2$ (respectively). We require that these sets are non-empty, that is $\frac{q-3}4,\frac{q+1}4\ge 1$. This implies that $q\ge 7$.

Note that for any square $x\in\Fq^*$, the square roots are $\pm\sqrt x$, and as $-1$ is a non-square, one of the roots is a square and one of the roots is a non-square. As $X,Y$ are squares, we can suppose $x,y$ are the roots of $X,Y$ (respectively) which are squares.  Define $u=\sqrt{\frac yx}$ and $h^2=-xy$. Note that this gives the requirement that $h^2$ is a non-square in $\Fq$.

We now check that this satisfies the requirements. We have $(-h^2/u^2+1)(-h^2u^2+1)$ is equal to
\[ \left(xy\frac xy +1\right)\left(xy\frac yx+1\right)=
(x^2+1)(y^2+1)=(X+1)(Y+1)
\]
which we have chosen to be a non-square, as required.

We have already shown that for such a choice, $h^2\in\Fq$ is a non-square. By definition of $X,Y$, $X\not=Y$ and $u\ne 0$. Now consider the condition $u^4\ne 1$. Suppose this does not hold, so $(u^2-1)(u^2+1)=0$. Firstly, consider the first term. If $u^2-1=0$, then we have $(X+1)(Y+1)=(-h^2+1)(-h^2+1)$, contradicting $X\ne Y$. Now consider the second term. As $-1$ is a non-square, the equation $x^2+1=0$ has no solutions over $\Fq$. Thus the condition $u^4\ne 1$ is satisfied, completing the proof in this case.

In summary, we have shown above that for all $q\ge 7$  a BM-special Triple O'Nans exists for $a$ non-square in the conic unital.
\end{proof}

 \subsection{Existence of BM-special Triple O'Nans for $q$ odd, $a$ square}

 It remains to show the existence of Triple O'Nans when $a$ is a square.

\subsubsection{BM-special Triple O'Nans for $q$ odd, $a$ square, $q\equiv 1\pmod 4$}

Without loss of generality we consider $\Uab$ with $a=1$ and $b=b_1e$ for some $b_1\in\Fq$ where $e=g^{\frac{q+1}2}$ for $g$ is a generator of $\Fqq^*$, so $\{1,e\}$ is a basis of $\Fqq$ over $\Fq$ and $w=e^2\in\Fq$ is a generator of $\Fq^*$.

\begin{theorem}\Label{thm:asq14}
  Suppose $q$ is odd  and $q\equiv 1\pmod 4$.   Let $h^2$ be any non-square in $\Fq$. Then there exists {in ${\mathcal U}_{1,b_1e}$ a}  BM-special Triple O'Nan having parameters $(1,b_1e,x,k,h,s,t)$ with $k=k_0+k_1e$, where $k_0 =-b_1e\rd\in\Fq,\ k_1=b_1\in\Fq$ and $x=\frac 1k$
  and $s,t\in\Fq^*$ with
  \begin{equation}
  t=\textstyle\frac{2h^2(\pm\sqrt{-1}-1)}{h^2+1},\ s=\pm\sqrt{-1}t.\Labele{eqn:stdf}
  \end{equation}

\end{theorem}
\begin{proof}
  Note that $q\equiv 1\pmod 4$ is a necessary and sufficient condition for $\sqrt{-1} \in\Fq$. Now   $s^2+t^2=0$, and as $a=1$ and $x=\frac 1k$ we have $\Delta=ax^2k^2=1$ . Further $\Theta=\sqrt{ (\frac\Delta a)^{q+1}}=1$. As $h^2\in\Fq$ is a non-square, it follows that $h=h_1e$ for some $h_1\in\Fq$ and $h^q=-h$ and $h^{q+1}=-h^2$.

  We need to show that $(1,b_1e,x,k,h,s,t)$ satisfies the BM-special Triple O'Nan equations (\ref{eqn:P1}) to (\ref{eqn:N1}). We make a few more calculations, in general using (\ref{eqn:WUdef})
  \[
W^2=\frac{h^2(s-t)^2}{(s-t\rd)^2},\quad W^{2q}=\frac{h^2(s-t)^2}{(s-t\rd)^{2q}}
\]
\[
\boxed{W^2}=W^2-W^{2q}=\frac{h^2(s-t)^2}{(s-t\rd)^{2(q+1)}}(2st)(\rd-\rd^q)==\frac{h^2(s-t)^2}{(s-t\rd)^{2(q+1)}}(2st)(2h)
\]
Further,
\begin{equation}
\boxed{kW}=kW-(kW)^q=\frac{(s-t)^2}{(s-t\rd)^{q+1}}\rd^{q+1}(2b_1e)=\frac{(s-t)^2}{(s-t\rd)^{q+1}}(-h^2)(2b_1e)\Labele{eqn:kW1}
\end{equation}
and similarly for $h^2\in\Fq$

\[
\boxed{U}=U-U^q=\frac{-st(s-t)}{(s-t\rd)^{q+1}}(2\rd).
\]
As $\Delta=1$ and $h^2\in\Fq$ we have $\boxed\Delta=\boxed{\Delta h^2}=0$ and so the BM-special Triple O'Nan equations can be simplified to:
\begin{eqnarray}
  2b_1e&=&k-k^q\Labele{eqn:101}\\
  2b_1e(-h^2)&=&k+k^q\Labele{eqn:102}\\
  \boxed{W^2}+2b_1eW^{q+1}&=&kW-(kW)^q+\boxed{U}\Labele{eqn:103}\\
  \boxed{V^2}+2b_1eV^{q+1}&=&kV-(kV)^q+\boxed{Z}\Labele{eqn:104}
\end{eqnarray}
(\ref{eqn:101}) and (\ref{eqn:102}) hold from the definition of $k_0,k_1$.

We now show that (\ref{eqn:103}) holds. Using (\ref{eqn:WUdef}), with $h^2\in\Fq$
\[
W^{q+1}=\frac{\rd^{q+1}(s-t)^2}{(s-t\rd)(s-t\rd^q)}=\frac{-h^2(s-t)^2}{(s-t\rd)^{q+1}}
\]
and so $2b_1eW^{q+1}=\boxed{kW}$. So  (\ref{eqn:103}) holds if we can show $\boxed{{W}^2}-\boxed{U}=0$.
\begin{eqnarray*}
  \boxed{W^2}-\boxed{U}&=&\frac{h^2(s-t)^2}{(s-t\rd)^{2(q+1)}}(2st)(2\rd)-\frac{-st(s-t)}{(s-t\rd)^{q+1}}(2\rd)\\
  &=&\frac{2h^2 st (s-t)}{(s-t\rd)^{2(q+1)}}\left(2h^2(s-t)+(s^2-h^2t^2)\right)\\
  &=&\frac{2h^2 st (s-t)}{(s-t\rd)^{2(q+1)}}\left(2h^2t(\pm \sqrt{-1}-1)-t^2(1+h^2)\right)\\
  &=&\frac{2h^2 st (s-t)}{(s-t\rd)^{2(q+1)}}[t(1+h^2)]\left(\frac{2h^2(\pm\sqrt{-1}-1)}{h^2+1}-t\right)
\end{eqnarray*}
which, by  (\ref{eqn:stdf}), is zero. Similarly  (\ref{eqn:104})
\begin{eqnarray*}
  \boxed{V^2}-\boxed{Z}&=&\frac{2h^2 st (t-s)}{(t-sh)^{2(q+1)}}\left(2h^2(t-s)+(t^2-h^2s^2)\right)
\end{eqnarray*}
is also zero.

\end{proof}

%\newpage
\subsubsection{BM-special Triple O'Nans for $q$ odd, $a$ square, $q\equiv 3\pmod 4$}
For $a$ is square in $\Fqq$ we may, without loss of generality, consider the BM unital $\Uab $ with $a=1$ and $b=b_1e$ for some $b_1\in\Fq$. First we note that, for $q\equiv 3\pmod 4$, $-1$ is not a square in $\Fq$ and so $-w$ is a square in $\Fq$. Let $h=\sqrt{-1}\in \Fqq \backslash \Fq$ be such that $h^2=-1$ and $\theta\in\Fqq$ be such that $\theta^2=e$. Note that, as $q\equiv 3\pmod 4$, $\frac{q+1}2$ is even and so $e$ is a square. Thus  $\Theta=\theta^{q+1}=\sqrt{-w} (=\frac e{\sqrt{-1}})$.
\begin{theorem}\Label{lem:q3asqONAN}
  Suppose $q$ is odd  and $q\equiv 3\pmod 4$.    Let $k\in\Fqq$ be defined by $k=k_0+k_1e$ with $k_0,k_1\in\Fq$ where
  \begin{equation}
    k_0=\sqrt{-w}(-1+b_1\sqrt{-w}),\quad
    k_1=1+b_1\sqrt{-w}.\Labele{eqn:kkdef}
   \end{equation}
Let $x=\theta/k$.

  Then there exists {in ${\mathcal U}_{1,b_1e}$ a} BM-special Triple O'Nan with parameters $(1,b_1e,x,k,h,s,t)$, if there exists  $s,t$ in $\Fq^*$, with $s^2\ne t^2$, satisfying
  \begin{equation}
    -\frac{s^2+t^2}{2(s+t)}=\sqrt{-w}\Labele{eqn:stdf1}.
  \end{equation}

\end{theorem}
\begin{proof}   We have
  \[
   h =\textstyle\frac1{\sqrt{-w}} e,\quad
   h ^q=-\sqrt{-1}=-h,\quad h ^{q+1}=1,\quad\textstyle \frac1{\sqrt{-1}}=\sqrt{-w}\ \frac 1e.
   \]
 As $x=\theta/k$,  we have $\Delta=ax^2k^2=1\cdot\theta^2=e$.

  Consider the BM-special Triple O'Nan equations:
\begin{eqnarray}
  \boxed{\Delta}+2b_1e\Theta&=&\boxed k\Labele{eqn:201}                    \\
  \boxed{\Delta h^2}+2b_1e\Theta h^{q+1}&=&\boxed{kh}   \Labele{eqn:202}      \\
  \boxed{\Delta W^2}+2b_1e\Theta W^{q+1}&=&\boxed{kW}+\boxed{U} \Labele{eqn:M3}   \\
  \boxed{\Delta V^2}+2b_1e\Theta V^{q+1}&=&\boxed{kV}+\boxed{Z}\Labele{eqn:N3}
\end{eqnarray}
In our case we have $h^2=-1\in\Fq$, $\Delta=e$ and $\Theta=\sqrt{-w}$ so (\ref{eqn:201}) and (\ref{eqn:202}) become
\begin{eqnarray}
  2e+2b_1e\sqrt{-w}&=&k-k^q\ =\ 2k_1e                          \Labele{eqn:P3}\\
  -2e+2b_1e\sqrt{-w}&=&\sqrt{-1}(k+k^q)\ = \ \sqrt{-1}\ 2k_0     \Labele{eqn:Q3}
\end{eqnarray}
Using the definition of $k$ from (\ref{eqn:kkdef}), it is straightforward to check that  (\ref{eqn:P3}) and (\ref{eqn:Q3}) hold. It remains to show that (\ref{eqn:M3}) and (\ref{eqn:N3}) hold.

Consider (\ref{eqn:M3}). Firstly
\begin{equation}
(s-th )^{q+1}=(s-th )(s+th )=s^2-h^2t^2=s^2+t^2.\Labele{eqn:std}
\end{equation}
We have
\begin{eqnarray*}
  \boxed{  \Delta W^2}&=&e(W^2+W^{2q})\\
  &=&e\left(\frac{-(s-t)^2}{(s-th )^2}+\frac{-(s-t)^2}{(s-th )^{2q}}\right)\\
  &=&\frac{-e(s-t)^2}{(s-th )^{2(q+1)}}\left((s^2-t^2+2st\sqrt{-1})+(s^2-t^2-2st\sqrt{-1})\right)\\
  &=&\frac{-2e(s-t)^2(s^2-t^2)}{(s-th )^{2(q+1)}}.
\end{eqnarray*}
and
\begin{eqnarray}
  2b_1e\Theta W^{q+1}=2b_1e\sqrt{-w}\frac{(s-t)^2}{(s-th )^{q+1}}\Labele{eqn:b1}
\end{eqnarray}
and
\begin{eqnarray*}
  \boxed{kW}&=&\frac{k\sqrt{-1}(s-t)}{s-th }-\frac{k^q(-\sqrt{-1})(s-t)}{(s-th )^q}\\
  &=&(s-t)\sqrt{-1}\left(\frac k{s-th }+\frac {k^q}{(s-th )^q}\right)\\
  &=&\frac{(s-t)\sqrt{-1}}{(s-th )^{q+1}}\left(k(s+th )+k^q(s-th )\right)\\
  &=&\frac{(s-t)\sqrt{-1}}{(s-th )^{q+1}}\left(s(k+k^q)+th (k-k^q)\right)\\
  &=&\frac{(s-t)\sqrt{-1}}{(s-th )^{q+1}}\left(\frac s{\sqrt{-1}} 2e(-1+b_1\sqrt{-w})+t\sqrt {-1}\ 2e(1+b_1\sqrt{-w})\right)\\
  &=&\frac{(s-t)}{(s-th )^{q+1}}\left(s2e(-1+b_1\sqrt{-w})-t2e(1+b_1\sqrt{-w})\right)\\
  &=&\frac{(s-t)}{(s-th )^{q+1}}\left(-2e(s+t)\right) + \frac{(s-t)}{(s-th )^{q+1}}\left(2eb_1\sqrt{-w}(s-t)\right)\\
  &=&\frac{(s-t)}{(s-th )^{q+1}}\left(-2e(s+t)\right) +2b_1e\Theta W^{q+1}.
\end{eqnarray*}
by (\ref{eqn:b1}), and
\begin{eqnarray*}
  \boxed{U}&=&\frac{st(1-h )}{s-th }-\frac{st(1+h )}{(s-th )^q}\\
  &=& -\frac{2st\sqrt{-1}(s-t)}{(s-th )^{q+1}}.
\end{eqnarray*}
Using (\ref{eqn:std}) we have
\begin{eqnarray*}
  &&\boxed{\Delta W^2}+2b_1e\Theta W^{q+1}-(\boxed{kW}+\boxed{U})\\
  &=&\frac{-2e(s-t)^2(s^2-t^2)}{(s-th )^{2(q+1)}}+2b_1e\Theta W^{q+1}\\
  &&\quad\quad -\ \left(\frac{(s-t)}{(s-th )^{q+1}}\left(-2e(s+t)\right) +2b_1e\Theta W^{q+1}-\frac{2st\sqrt{-1}(s-t)}{(s-th )^{q+1}}\right)\\
  &=&\frac{-2e(s-t)^2(s^2-t^2)}{(s-th )^{2(q+1)}}-\left(\frac{(s-t)}{(s-th )^{q+1}}\left(-2e(s+t)\right) -\frac{2st\sqrt{-1}(s-t)}{(s-th )^{q+1}}\right)\\
  &=&\frac {s-t}{(s-th )^{2(q+1)}}\left(-2e(s-t)(s^2-t^2)-(s^2+t^2)(-2e(s+t))+2st\sqrt{-1}(s^2+t^2)\right)\\
  &=&\frac {s-t}{(s-th )^{2(q+1)}}\left(2e(s+t)(2st)+2st\sqrt{-1}(s^2+t^2)\right)\\
  &=&\frac {2st(s-t)}{(s-th)^{2(q+1)}}\left(2e(s+t)+\sqrt{-1}(s^2+t^2)\right)\\
  &=&\frac {2st(s-t)}{(s-th)^{2(q+1)}}[2\sqrt{-1}(s+t)]\left(\frac e{\sqrt{-1}}+\frac{s^2+t^2}{2(s+t)}\right)\\
  &=&0
\end{eqnarray*}
by (\ref{eqn:stdf1}). Note that $2e(s+t)+\sqrt{-1}(s^2+t^2)$ is symmetrical in $s$ and $t$, and so  (\ref{eqn:N3}) also holds.
\end{proof}

\begin{theorem}\Label{thm:q3asqONAN}
   Suppose $q$ is odd  and $q\equiv 3\pmod 4$. Then there exists BM-special Triple O'Nan configurations  in any BM unital $\Uab $ with $a$ square if $q>3$.
\end{theorem}

\begin{proof}
  From Theorem~\ref{lem:q3asqONAN} we need to show there exists $s,t\in\Fq^*$ with
  \[
\sqrt{-w}=-\frac{s^2+t^2}{2(s+t)}
\]
and $s^2\ne t^2$.
Now $\sqrt{-w}=-\frac{s^2+t^2}{2(s+t)}$ and $s^2=t^2$ implies $s=t=-2\sqrt{-w}$ since $s+t\not=0$. Put $X=\frac{s}{2\sqrt{-w}}$ and $Y=\frac{t}{2\sqrt{-w}}$. Then we require $X^2+Y^2+X+Y=0$ and $X,Y \not= 0,-1$. That is $X(X+1)=-Y(Y+1)$ and non-zero. Since $-1$ is a non-square then $X,X+1$ are both square or both non-square and $Y,Y+1$ has one square and one non-square. Note that this automatically excludes $X=Y$. The number of such $X,Y$ may be expressed as
\begin{eqnarray*}
  &&(0,0)_2(0,1)_2+(0,0)_2(1,0)_2+(1,1)_2(1,0)_2+(1,1)_2(0,1)_2\\
  &=&\left((0,0)_2+(1,1)_2\right)\left((1,0)_2+(0,1)_2\right)\\
  &=&\left(\frac{q-3}2\right)\left(\frac{q-1}2\right)
\end{eqnarray*}
which is $>0$ for $q>3$.
\end{proof}

  \section{Further properties of BM-special Triple O'Nans}

\begin{figure}[ht]
\centering
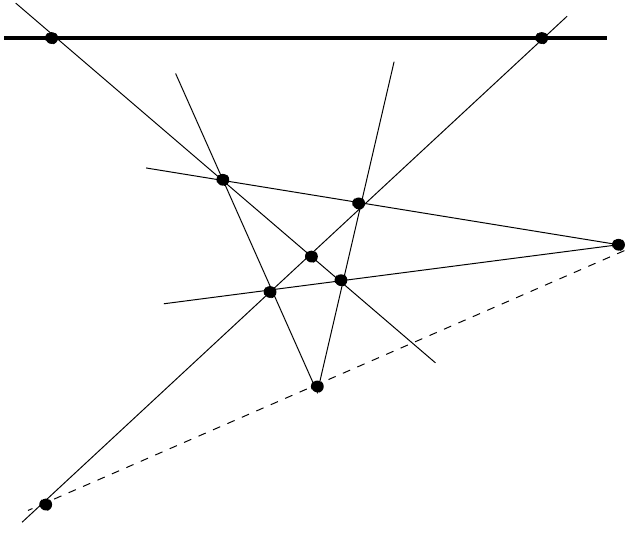
\caption{Extended Triple O'Nan}\label{fig:extendonan}
\end{figure}

\begin{lemma}\Label{lem:EF}
  Consider a BM-special Triple O'Nan using the notation of Section~\ref{sec:typea}. Then the point $F=MN\cap XY$ is always a point of the unital.
\end{lemma}
\begin{proof}
  From Section~\ref{sec:TypeA} we have
\[
M=\left(kxW,kW+U,1\right),\quad
N=\left(kxV,kV+Z,1\right).
\]
and so
\begin{equation}
  MN=[kW+U-(kV+Z),\ kx(V-W),\ kx(WZ-VU)].\Labele{eqn:MN}
\end{equation}
using the definitions of $W,U,V,Z$ from equations (\ref{eqn:WUdef}) and (\ref{eqn:VZdef}) we obtain
\begin{eqnarray*}
  WZ-VU&=& \frac{st(s-t)2h(1-h)}{(s-th)(t-sh)}\\
  Z-U&=& \frac{st(s-t)(1-h)(1+h)}{(s-th)(t-sh)}\\
  W-V&=& \frac{(s-t)(s+t)h(1-h)}{(s-th)(t-sh)}
\end{eqnarray*}
and so
\begin{eqnarray*}
  MN&=&[k(s-t)(s+t)(1-h)-st(s-t)(1-h)(1+h),\\
    &&\quad -kx(s-t)(s+t)(1-h),\ \ kxst(s-t)2h(1-h)]\\
  &=&[kh(s+t)-(1+h)st,\ -kxh(s+t),\ 2kxhst]
\end{eqnarray*}
which, by inspection, meets $[1,0,0]$ in the point $F=(0,f,1)$ where  $f=\frac{2st}{s+1}\in\Fq$. As $s,t\in\Fq$ it follows that $F$ is a point of the unital.
\end{proof}

We now calculate $E=MN\cap PQ=MN\cap [-1,x,0]$, for $f$ defined above
\[
E=\left( (2khst)x,\ 2khst,\  (h+1)(s+t)\right)=
\left( \frac{(khf)x}{h+1},\ \frac{khf}{h+1},\ 0\right)
\]
and conjecture that if $E$ is a point of the unital, then $a$ is a square.

\section{Conclusion}
This paper has made a contribution to support Conjecture~\ref{conj:big}.

The O'Nans of Feng and Li has the property that one of its {lines} contains the special point $T_\infty=(0,1,0)$, and we showed that this does not extend to a Triple O'Nan. However, we give a construction of BM-special O'Nans. From Section~\ref{sec:typea}, we see that a Triple O'Nan contains three O'Nans, and hence that {there exist O'Nans}  with the property that none of {their lines contain} the special point $T_\infty$.

As far as the authors are aware, there are no BM-ordinary Triple O'Nan constructions for BM unitals in the the case $q$ is odd.

For the interested reader, the authors have done further work on Triple O'Nans for BM unitals in the case $q$ even. {That work will be posted on Arxiv}.

\end{document}